\documentclass[a4paper]{article}
\usepackage[english]{babel}
\usepackage[utf8x]{inputenc}
\usepackage[T1]{fontenc}
\usepackage{skull}
\usepackage{enumerate}
\usepackage[a4paper,top=2cm,bottom=2cm,left=3cm,right=3cm,marginparwidth=1.75cm]{geometry}
\usepackage{subfig}
\usepackage{caption}
\usepackage{amsmath}
\usepackage{graphicx}
\usepackage[colorinlistoftodos]{todonotes}
\usepackage[colorlinks=true, allcolors=blue]{hyperref}

\usepackage{amsthm}
\usepackage{amsfonts}
\usepackage{amssymb}
\usepackage{epic}
\usepackage{graphicx}
\usepackage{epstopdf}
\usepackage{epsfig}
\usepackage{authblk}

\usepackage{array}
\usepackage{tabularx}
\usepackage{multirow}
\usepackage{longtable}
\usepackage{pbox}

\usepackage{geometry} % Required for adjusting page dimensions

\newcolumntype{L}[1]{>{\raggedright\let\newline\\\arraybackslash\hspace{0pt}}m{#1}}
\newcolumntype{C}[1]{>{\centering\let\newline\\\arraybackslash\hspace{0pt}}m{#1}}
\newcolumntype{R}[1]{>{\raggedleft\let\newline\\\arraybackslash\hspace{0pt}}m{#1}}

\newtheorem{thm}{\textbf{Theorem}}
\makeatletter
\@addtoreset{proofpart}{thm}
\makeatother
\newtheorem{remark}[thm]{Remark}

\newtheorem{pro}{\textbf{Proposition}}

\newtheorem{lemma}{Lemma}%[section] 

\newtheorem{corollary}[lemma]{Corollary}

\newtheorem{definition}[lemma]{Definition}

\newcommand{\ignore}[1]{}
\newcommand{\Z}{\mathbb{Z}}
\newcommand{\C}{\mathbb{C}}
\newcommand{\sgn}{\mathrm{sgn}}
\newcommand{\id}{\mathrm{id}}

\def\Z{\ensuremath\mathbb{Z}}
\def\s{\ensuremath\sigma_0}
\def\ss{\ensuremath\sigma_1}
\def\sss{\ensuremath\sigma_{\infty}}

\title{Shabat Polynomials and Monodromy Groups of Trees Uniquely Determined by their Passport}
\author{
Naiomi Cameron\\
Lewis \& Clark College\\
\texttt{ncameron@lclark.edu}

\and 

Mary Kemp\\
Occidental College\\
\texttt{maroldkemp@gmail.com} 

\and

Susan Maslak\\
Ave Maria University\\
\texttt{susan.m.maslak@gmail.com} 

\and

Gabrielle Melamed\\
University of Hawaii at Manoa\\
\texttt{gmelamed@hawaii.edu} 

\and

Richard A. Moy\\
Willamette University\\
\texttt{rmoy@willamette.edu} 

\and

Jonathan Pham\\
University of California - Irvine\\
\texttt{jonatdp1@uci.edu} 

\and 

Austin Wei\\
Pomona College\\
\texttt{abw22014@mymail.pomona.edu}
}

\begin{document}
\maketitle

\begin{abstract}
A dessin d'enfant, or dessin, is a bicolored graph embedded into a Riemann surface. Acyclic dessins can be described analytically by pre-images of certain polynomials, called Shabat polynomials, and also algebraically by their monodromy groups, that is, the group generated by rotations of edges about black and white vertices. In this paper we investigate the Shabat polynomials and monodromy groups of planar acyclic dessins that are uniquely determined by their passports.

\end{abstract}

\section{Introduction}\label{introduction}

Popularized by Grothendieck in his \textit{Esquisse d'un Programme}, the theory of {\em dessins} reaches across and connects multiple disciplines, including graph theory, topology, geometry, algebra and complex analysis. Our motivation for this paper is rooted in one of the fundamental questions in the theory of dessins -- that is, how to distinguish classes of dessins by means of topological, algebraic and/or combinatorial invariants. In this paper, we focus our attention to this question by studying dessins which qualify as trees, the permutation groups generated by rotations of their edges (otherwise known as monodromy groups) and the special polynomials, known as Shabat polynomials, to which they correspond. Monodromy groups in particular are of interest because, as algebraic invariants, their group structure suggests certain properties of the corresponding dessins. 

Our main objective in this paper is to determine the Shabat polynomials (up to isomorphism) and monodromy groups corresponding to every known planar connected acyclic dessin with a passport of size one, the complete list of which was given in \cite{shabat_plane_1994}. We begin in Section \ref{introduction} by providing the basic definitions and notation needed to describe this class of dessins, as well as some necessary background about Shabat polynomials and wreath products; readers already acquainted with these subjects may wish to skip directly to Section \ref{mainresults} for a summary of results. In Section \ref{shabatpolys} we provide a unique (up to isomorphism) Shabat polynomial for each passport of size one corresponding to a (planar) bicolored tree; in Section \ref{monodromygroups} we provide the monodromy groups for such passports.

\subsection{Dessins, Shabat Polynomials and Monodromy Groups}
We begin by providing a terse exploration of the object known as a {\em dessin}. For more detailed and comprehensive literature on the subject, see \cite{shabat_plane_1994}. For the purposes of this paper, we begin with the observation that dessins may be realized by meromorphic functions known as {\em Belyi maps}.

\begin{definition}
 Let $X$ be a Riemann surface.  A {\bf Belyi map} is a meromorphic function $F:X\to \mathbb{P}^{1}({\mathbb{C}})$ that is unramified outside of $\{0,1,\infty\}$. That is, all critical values of $F$ are contained in $\{0,1,\infty\}.$  Here we may consider $\mathbb{P}^{1}({\mathbb{C}})$ as just $\mathbb{C}\cup\{\infty\}.$
%A \emph{dessin d'enfant} is a bicolored graph embedded into a Riemann surface, which induces an ordering of edges around the vertices.
\end{definition}

\begin{remark}
In this paper, we will only be concerned with the case where $X=\mathbb{P}^{1}({\mathbb{C}})$, so the reader may assume throughout that $X$ is the Riemann sphere.
\label{firstremark}
\end{remark}
%A {\em dessin d'enfant} (or {\em dessin} for short) $D$ can be recovered from a Belyi map $F$ in the following way.

\noindent Grothendieck's notion of a {\em dessin d'enfant} is a way to combinatorially characterize Belyi maps.  If $F$ is a Belyi map, then $F^{-1}([0,1])$, that is, the preimage of the interval $[0,1],$ has the structure of a bicolored connected graph embedded in $X$.  The basic structure of the bicolored graph $\Delta_F$ associated with a Belyi map $F$ is given when we identify $F^{-1}(0)$ as the set of black vertices, $F^{-1}(1)$ as the set of white vertices, $F^{-1}((0,1))$ as the set of edges and $F^{-1}(\infty)$ as the set of faces. Note that the degrees of the black and white vertices of $\Delta_F$ correspond to the multiplicities of the roots of $F$ and $F-1$, respectively.  

These structure of $\Delta_F$ can be captured by the notion of a {\em dessin}, the relatively simple combinatorial characterization given by Grothendieck.

\begin{definition}
 A {\bf dessin d'enfant} or {\bf dessin} is a connected bicolored graph equipped with a cyclic ordering of the edges (oriented counterclockwise) around each vertex.  
\end{definition}

Given a Belyi map $F$, it is not difficult to use the procedure described above to visualize the dessin $\Delta_F$ to which $F$ corresponds.  However, recovering a Belyi map from a given dessin is a much more difficult proposition. Given a dessin $\Delta_F$, a corresponding Belyi map $F$ can be determined (uniquely up to isomorphism over $\mathbb{C}\cup \{\infty\}$) by considering the degrees of the vertices of $\Delta_F$ and the resulting system of polynomial equations involving roots and poles of $F$.

\begin{remark}
The choice made in Remark \ref{firstremark} will guarantee that the dessin recovered from a Belyi map $F$ qualifies as a planar graph. Further, in order to guarantee that the dessin resulting from $F$ is a tree (i.e., having only one face), we could require $F$ to be a polynomial.
\end{remark}

\begin{definition}A {\bf Shabat polynomial} is a polynomial $F:{\mathbb{C}}\rightarrow {\mathbb{C}}$ whose critical values are contained in $\{0,1\}$.
\end{definition}

That is, a Shabat polynomial is a Belyi map which has only one pole (which is at infinity); hence, its corresponding dessin will be a tree. (Shabat polynomials can be defined more broadly as in \cite{shabat_plane_1994} as a generalized Chebyshev polynomials which have at most two critical values. Without loss of generality, we choose in this paper to identify the two critical values $0$ and $1.$)

\begin{definition}
We say that two Shabat polynomials $F,G$ are isomorphic if there exist $\alpha,\beta\in\C$ such that $F(z)=G(\alpha z+\beta)$.
\end{definition}

Assume we have a dessin which is a tree and we label the edges with the numbers $1,2,\dots,n.$  We can associate the dessin with a pair of permutations $\sigma_0$, $\sigma_1\in S_n$, where $n$ is number of edges, such that the cycles of $\sigma_0$ correspond to the cyclic ordering read counterclockwise of the edges around the black vertices and the cycles of $\sigma_1$ correspond to the ordering of the edges around the white vertices.  For example, see Figure \ref{treeex}, where we have a bicolored tree, whose edges are labeled $1,2,\dots, 7$ inducing a pair of permutations $\sigma_0,\sigma_1\in S_7$ associated with the black and white vertices, respectively. In general, by $\sigma_{0}$ (respectively, $\sigma_1$), we mean the product of the cycle permutations associated with the edges about all of the black (respectively, white) vertices. The group that $\sigma_0$ and $\sigma_1$ generate is a central focus of this paper.

\begin{figure}
\begin{center}
\includegraphics[width=0.35\textwidth]{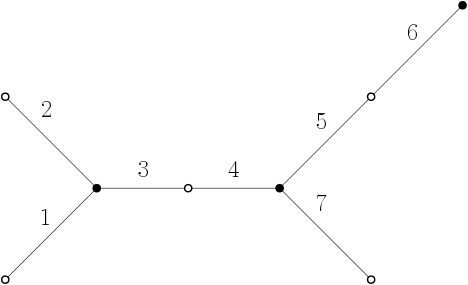} \\
\end{center}
\caption{A dessin determined by the pair of permutations $\sigma_0 = (1,3,2)(4,7,5)$, $\sigma_1 = (3,4)(5,6)$ whose monodromy group $\langle \sigma_0,\sigma_1\rangle$ is isomorphic to $GL_3(\mathbb{F}_2)$, a transitive subgroup of $S_7$.}
\label{treeex}
\end{figure}

\begin{definition}
The {\bf monodromy group} of a dessin is $\langle \sigma_{0} ,\sigma_{1},\sigma_{\infty} \rangle$, where $\sigma_{\infty}$ is such that $\sigma_{0} \sigma_{1} \sigma_{\infty} = 1$. 
\end{definition}

We remark that since $\sigma_{\infty} = (\sigma_{0} \sigma_{1})^{-1}$, we may remove it from the generating set for the monodromy group, but we keep it in the definition to be consistent with the wider literature on this subject, which goes well beyond the consideration of Shabat polynomials. 
For the remainder of the paper, when we refer to the generators of the monodromy group, we are talking about $\sigma_0$ and $\sigma_1$. 
When a dessin is connected, its monodromy group will be a transitive subgroup of $S_n$, where $n$ is the number of edges in the dessin.

%\noindent The ordering assigned to each edge encourages us to consider the permutations created by rotating each edge counterclockwise. 

To every dessin, we may associate an invariant known as its {\em passport}. While there are several known varieties of passports discussed in the literature, the following definition is sufficient for the purposes of this paper.

\begin{definition}
 The {\bf passport} of a dessin is a pair of lists $[b_1, b_2,\dots, b_k;w_1,w_2,\dots, w_\ell]$ where $b_1, b_2,\dots, b_k$ are the degrees of the black vertices and $w_1,w_2,\dots, w_\ell$ are the degrees of the white vertices.
 \end{definition}
 
 \begin{remark}\label{remark:passportnotation}
 Note that the lists $b_1,b_2,\dots,b_k$ and $w_1,w_2,\dots, w_\ell$ are both partitions of $n$, where $n$ is the number of edges, and these two partitions correspond to the cycle type of $\sigma_0$ and $\sigma_1$, respectively. Within the passport notation, we utilize the commonly known exponential notation to express multiplicity in partitions of $n$.  For example, the expression $r^k$ appearing within the passport notation denotes $k$ vertices of degree $r$.  
 \end{remark}

While each dessin has a unique passport, one may ask how many distinct dessins (or equivalently non-isomorphic Shabat polynomials) are associated with a given passport. The \textbf{size} of a passport is the number of non-isomorphic dessins (equivalently, non-isomorphic Shabat polynomials) corresponding to that passport. Our focus in this paper will be narrowed to passports of size one which by definition admit unique dessins.

\subsection{Composition and Wreath Products}
A dessin with $n$ edges admits a monodromy group contained in the symmetric group $S_n$. We sometimes use the concept of tree composition to decompose a dessin into smaller dessins whose monodromy groups are subgroups of smaller symmetric groups. Composition will also help us compute new Shabat polynomials as it corresponds with the usual polynomial composition. It is an easy exercise in calculus to show that the composition of two Shabat polynomials is again a Shabat polynomial. 

Many of the dessins that we study can be constructed by a composition process given by Adrianov and Zvonkin \cite{adrianov_composition_1998}. 
Given two dessins, $P$ and $Q$, we begin the composition $P\star Q$ by first distinguishing two vertices of $P$ -- label them with a square and a triangle. The vertices of $Q$ will be pre-images of the square and triangle, so we mark every black vertex of $Q$ with a square and similarly every white vertex of $Q$ with a triangle. The process of composition is as follows:
\begin{enumerate}
\item Replace each edge of $Q$ with the union of the path from the square to triangle in $P$ along with every branch connected to that path.
\item Adjoin to each square (resp., triangle) vertex of $Q$, the union of every branch connected to the square (triangle) in $P$ except for the one in the path to the triangle (square). Do this as many times as the degree of the vertex.
\end{enumerate}
The resulting graph should resemble $n$ copies of $P$ arranged in the shape of $Q$, where $n$ is the number of edges of $Q$. We demonstrate this process in Figure \ref{PcircQ}. 
\begin{figure}
\centering
\subfloat[$P$, with two vertices marked square $\square$ and triangle $\triangle$]{\includegraphics[width=0.25\linewidth]{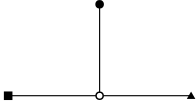}}
\hspace{0.5cm}
\subfloat[$Q$, with black vertices marked $\square$, white vertices marked $\triangle$]{\includegraphics[width=0.25\linewidth]{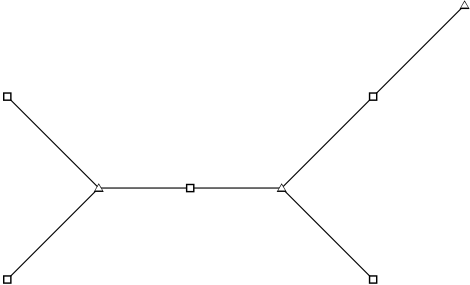}}
%\end{figure}
%\begin{figure}[h!]
\begin{center}
{\includegraphics[width=0.37\textwidth]{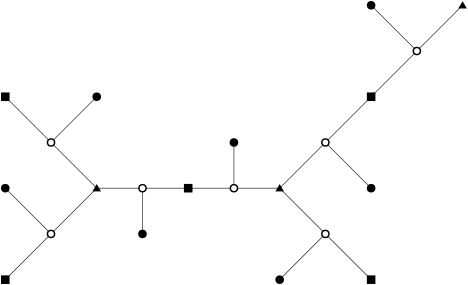}}
\caption{The composition $P \star Q$ of two dessins $P$, $Q$.}
\label{PcircQ}
\end{center}
\end{figure}

\begin{remark}
Let $G_P$, $G_Q$ denote the respective monodromy groups of $P$ and $Q$. According to a theorem of Adrianov and Zvonkin, the monodromy group of $P\star Q$ is a subgroup of $G_Q \wr G_P$, where $\wr$ denotes the wreath product \cite{adrianov_composition_1998}. 
\end{remark}

%Remember to motivate the choice of square and triangle with Shabat polynomial.
%fix
\noindent This process also gives a way to compute Shabat polynomials. If $p, q$ are the respective Shabat polynomials of $P, Q$ such that $p(0), p(1) \in \{0,1\}$ then the Shabat polynomial of $P \star Q$ is $p \circ q$ (where $\circ$ denotes the conventional composition of functions, i.e., $(f\circ g)(x)=f(g(x))$). Later on, when we compute Shabat polynomials of more complicated dessins, we will make extensive use of this fact.

We will often call upon the idea of the wreath product of groups to describe our monodromy groups. The composition process produces dessins whose monodromy groups as subgroups of wreath products, and while there are numerous examples for which the containment is proper, often equality of the groups is achieved. As far as the present authors can tell, the exact conditions that ensure equality are not known.  
\begin{definition}
Let $G\le S_d$ and $H$ be groups where let $d$ is a positive integer. Let $K$ be the direct product of $d$ copies of $H$. If $h=(h_1,\dots,h_d)\in K$, then we define the action of $\sigma\in G$ on $K$ by $\sigma\cdot h=(h_{\sigma(1)},\dots, h_{\sigma(d)})$. We say that the {\bf wreath product} of $H$ by $G$ is the semi-direct product $K \rtimes G$ with respect to the action above, and we denote this group $H\wr G$. 

%If $\psi: S_n \rightarrow Aut(K)$ is an injective homomorphism which sends the elements of $S_n$ to automorphisms permuting the $n$ factors of $K$, then $\phi=\psi \circ \varphi$ (read right to left) is a homomorphism from $G$ into $Aut(K)$ and we say that the \emph{wreath product} of $H$ by $G$ is the semi-direct product $K \rtimes G$ with respect to $\phi$ and is denoted $H\wr G$.
\end{definition}

In this paper, $G$ is typically $\mathbb{Z}_d$, the cyclic group of order $d$.

%more
%As an example, consider $H \wr \mathbb{Z}_2$. Then $\phi$ is given by $\phi(1) = (\pi: (g,h) \mapsto (h,g))$ for all $(g,h) \in H \times H$. 

\subsection{Summary of Results}\label{mainresults}
Table \ref{bigtable} below lists the passports which have size one and correspond to dessins which are trees, along with the associated monodromy groups and Shabat polynomials. It contains every such passport of size one, as asserted in \cite{shabat_plane_1994}. In Sections \ref{shabatpolys} and \ref{monodromygroups}, we argue that Table \ref{bigtable} lists the correct Shabat polynomials and monodromy groups. 

%[SHOULD WE PUT THE TABLE IN AN APPENDIX INSTEAD OF HAVING IT BROKEN UP HERE BETWEEN TWO PAGES?]

\begin{center}
\setlength\extrarowheight{7pt}
\begin{longtable}{|c|C{2.1in}|C{2in}|}
\hline
\label{bigtable}
{\bf Passport } &{\bf Shabat Polynomial} &  {\bf Monodromy Group} \\\hline
$[r;1^r]$ & $z^r$
& $\mathbb{Z}_r$\\\hline

$[2^r,1;2^r,1]$ &   $\displaystyle\frac{1+\cos\left((2r+1)\arccos(z)\right)}{2}$ & $D_{2(2r+1)}$\\\hline

$[2^r;2^{r-1},1^2]$ &  $\displaystyle \frac{1+\cos\left(2r\arccos(z)\right)}{2}$ & $D_{2(2r)}$\\\hline

\pbox{20cm}{\text{        }\newline
\newline
\newline
\newline
$[s^{r-1},t;r, 1^{(r-1)(s-1)+(t-1)}]^{*}$ \\
\newline
\newline
${}^{*}$\text{\footnotesize{for $r > 1$}}} & 
$\displaystyle (1-z)^t \left(\sum_{k=0}^{r-1}{\left(\frac{t}{s}\right)_k \frac{z^k}{k!}}\right)^{s}$  & $\begin{array}{ll} 
      \mathbb{Z}_r \wr \Z_s, \text{ if $s = t$} \\ 
      S_\frac{n}{d} \wr \Z_{d}, \text{ if $r$ even, $(\frac{n}{d},d)=1$} \\
      \widetilde{R}_d, \text{ if $r$ odd, $\frac{t}{d}$ even, $(\frac{n}{d},d)=1$} \\
      A_\frac{n}{d} \wr \Z_{d}, \text{ if $r, \frac{t}{d}$ odd, $(\frac{n}{d},d)=1$} \\
      \text{??,}\hspace{0.5in} \text{otherwise}\\
      \text{\footnotesize{where $n = s(r-1)+t$, $d = (s,t)$}} \\
   \end{array}$ 
   \\\hline

\pbox{20cm}{\text{        }\newline
\newline
\newline
\newline
$[r,t,1^{r+t-2};2^{r+t-1}]^{\ddagger}$ \\
\newline
\newline
${}^{\ddagger}$\text{\footnotesize{for $r,t > 1$}}} & $ 4z^r(1-z)^t \left(\sum_{j=0}^{r-1}{\binom{t-1+j}{t-1}z^j}\right)*\left(\sum_{j=0}^{t-1}{{ \binom{r-1+j}{r-1}}{   \binom{r+t-1}{r+j}    }(-1)^jz^j}\right)$ & $\begin{array}{cl}
	A_{2r-1} \times \Z_2 & r=t, \text{ $r$ odd}\\
      S_{2r-1} \times \Z_2 & r=t, \text{ $r$ even}\\
      A_{r+t-1} \wr \Z_2 & r\not=t, \text{ both odd} \\
      R_2 & r\not=t, \text{ both even} \\
      S_{r+t-1} \wr \Z_2 & r\not=t, \text{ else}
    \end{array}$
    
\\\hline

$[r^2,1^{4r-3};3^{2r-1}]$ & $-3\sqrt{3}\ i\ S_r(z)\left(1-S_r(z)\right) * \left(S_r(z)-\frac{1-i\sqrt{3}}{2}\right)$  & 
$\begin{array}{cl}
A_{2r-1} \wr \mathbb{Z}_3 & r \text{ odd} \\
R_3 & r \text{ even} \\
\end{array}$
\\\hline

$[3^3,1^5;2^7]$ & $ -\frac{4}{531441}(z-1) z^3 \left(2 z^2+3 z+9\right)^3 * \left(8 z^4+28 z^3+126 z^2+189 z+378\right)$ & $A_7\wr \mathbb{Z}_2$ \\\hline
\caption[k=1]{Shabat polynomials and monodromy groups for all passports producing exactly one tree. $D_{2n}$ denotes the dihedral group of order $2n$, and $R_m$ denotes the index 2 subgroup of $S_{\frac{n}{m}} \wr \mathbb{Z}_m$ such that for all $(\tau_1, \dots, \tau_m, g) \in R_m$, $\tau_1 \tau_2\cdots\tau_m$ is an even permutation. Here, $n$ refers to the number of edges of the dessin. The group $\widetilde{R}_d$ denotes the index $2^{d-1}$ subgroup of $S_{\frac{n}{d}} \wr \mathbb{Z}_d$ such that for all $(\tau_1, \dots, \tau_d, g) \in \widetilde{R}_d$, $\sgn(\tau_1)=\sgn(\tau_2)=\dots=\sgn(\tau_d)$. Notational note: $(x,y):=\gcd(x,y)$.}
\end{longtable}
\end{center}
%%

%%%%%%%%%%%%%%%%%%%%%%%%%%%%%%%%%%%%%%%%%%%%%%%%%%%%%%%%%%
%%%%%%%%%%%%%%%%%%%%%%%%%%%%%%%%%%%%%%%%%%%%%%%%%%%%%%%%%%
\section{Shabat Polynomials for Passports of Size One}\label{shabatpolys}

In this section, we summarize the list of Shabat polynomials (up to isomorphism) corresponding to dessins which are trees and have passports of size one. The complete list of passports for such dessins was given in \cite{shabat_plane_1994}. For the Shabat polynomials corresponding to these passports, we adopt the convention described in Remark \ref{remark:passportnotation}.

\begin{figure}
\begin{center}
\includegraphics[width=0.25\textwidth]{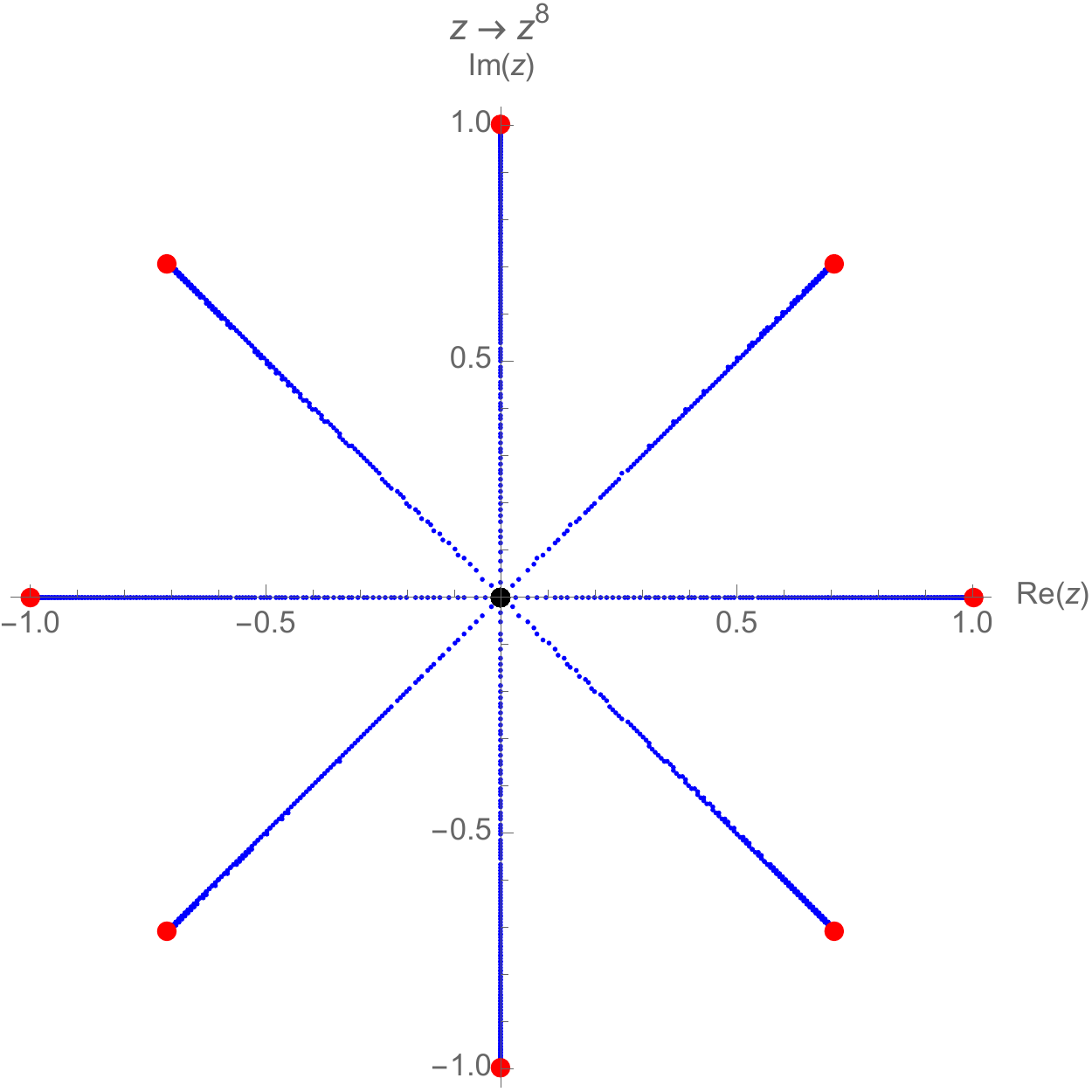} \\
\caption{The dessin with passport $[8;1^8]$.}
\end{center}
\end{figure}

\begin{pro}
The passports $[r;1^r]$, $[2^r, 1; 2^r, 1]$, $[2^r; 2^{r-1}, 1^2]$ have respective Shabat polynomials $z^r, \frac{1}{2}(1+\cos((2r+1)\arccos(z))), \frac{1}{2}(1 + \cos((2r)\arccos(z)))$, all unique up to isomorphism.
\end{pro}
\noindent This result is already well-known in the literature and can be found on pages 3-4 of \cite{shabat_plane_1994}.
%%%%%%%%%%%%%%%%%%%%%%%%%%%%%%%%%%%%%%%%%%%%%%%%%%%%%%%%%%

%PRO MEANS PROPOSITION
\begin{pro}[N. Adrianov, \cite{adrianov_2009}]
Up to isomorphism, the unique Shabat polynomial for the passport

$[s^{r-1}, t; r, 1^{(r-1)(s-1)+(t-1)}]$ is 
$$F(z)=(1-z)^t \left(\sum_{k=0}^{r-1}{\left(\frac{t}{s}\right)_k \frac{z^k}{k!}}\right)^s,$$
where $(a)_k=a(a+1)(a+2)\cdots(a+k-1)$ denotes the Pochhammer symbol.
\label{iii}
\end{pro}

\begin{figure}
 \begin{center}
    \includegraphics[width=1.5in]{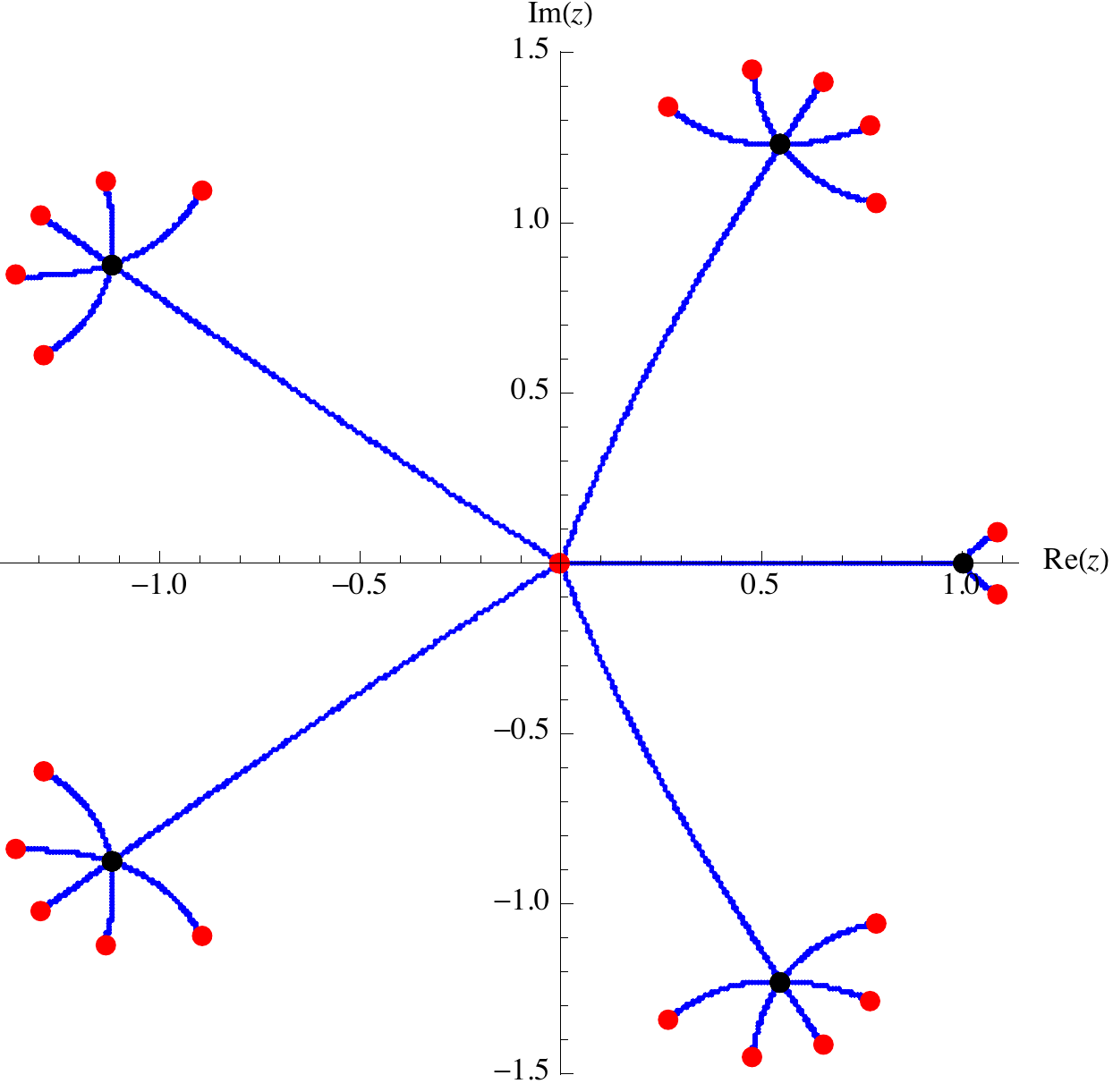}
  \end{center}
  \caption{The dessin obtained by the Shabat polynomial given in Proposition \ref{iii} when $s=6, r=5, t=3.$}
  \label{classiii}
\end{figure}

\noindent The proof for this proposition can be found in \cite{adrianov_2009}. %For the next passport, we impose the condition $r > 1$ since otherwise we would just have a $t$-star, the Shabat polynomial and monodromy group of which has already been discusseed.
%%%%%%%%%%%%%%%%%%%%%%%%%%%%%%%%%%%%%%%%%%%%%%%%%%%%%%%%%%

\begin{pro}\label{cleansplitstarprop}
Let $r>1.$  Up to isomorphism, the Shabat polynomial for the tree having passport $$[r,t,1^{r+t-2};2^{r+t-1}]$$ with a black vertex of degree $r$ located at $z=0$, a black vertex of degree $t$ located at $z=1$ is given by

$$ F(z)=4z^r \binom{r+t-1}{r}{_2F_1}(t-1,r;r+1;z)\left(1-(1-z)^tz^r\binom{r+t-1}{t-1}{_2F_1}(1,r+t;r+1;z)\right) $$ 
where $_2F_1$ is the hypergeometric function defined by $_2F_1(a,b;c;z)=\displaystyle\sum_{n=0}^{\infty}{\frac{(a)_n(b)_n}{(c)_n}}\frac{z^n}{n!}$.
\end{pro}

\begin{figure}
 \begin{center}
    \includegraphics[width=2in]{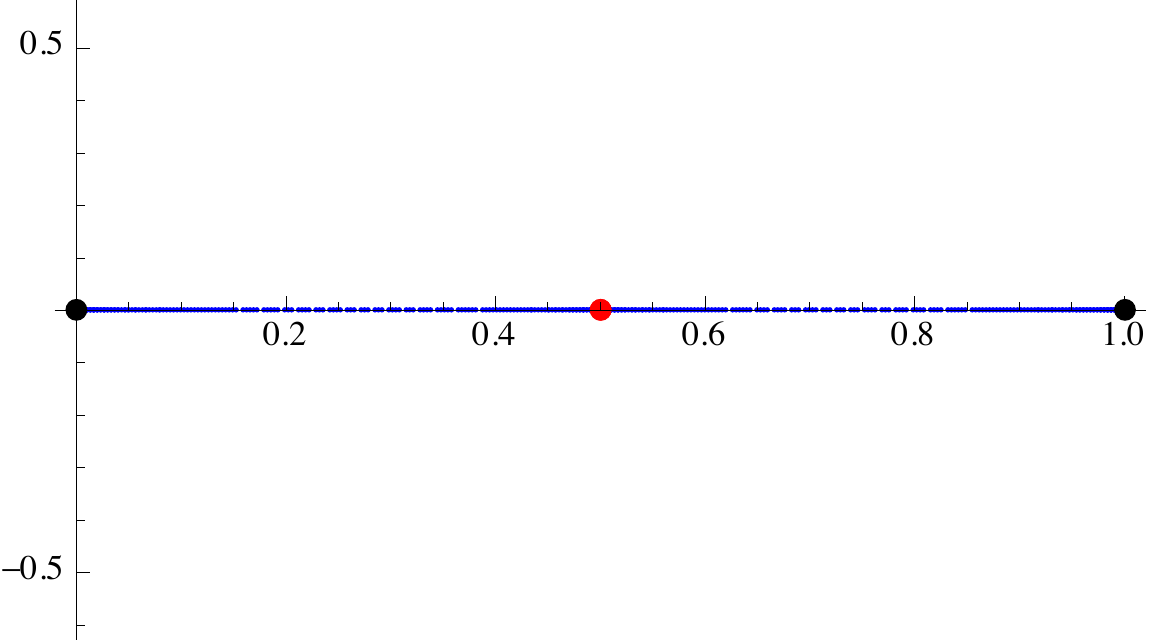}
  \end{center}
  \caption{The dessin (path graph) obtained by the Shabat polynomial $\beta(z)=4z(1-z).$}
    \label{cleanmap}
  \end{figure}
  
    \begin{figure}
 \begin{center}
    \includegraphics[width=1.5in]{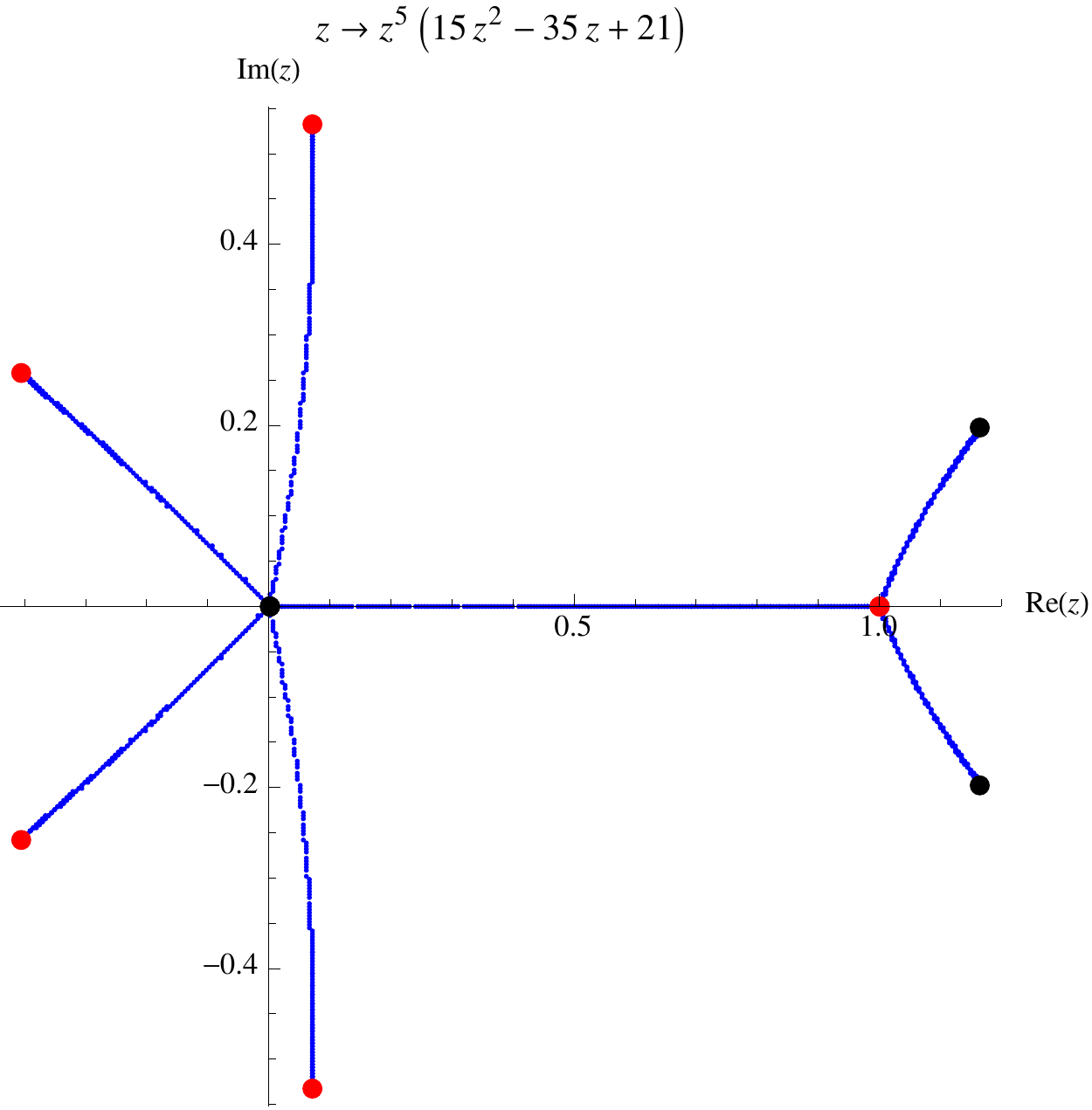}
  \end{center}
  \caption{The tree obtained by the Shabat polynomial in Proposition \ref{iii} where $s=1,r=3, t=5.$}
    \label{splitstar}
  \end{figure}

\begin{proof}
Let $S_{r,t}(z)$ denote the Shabat polynomial for the passport $[t,1^{r-1};r,1^{t-1}]$. By Proposition \ref{iii}, with $s=1$, $$S_{r,t}(z)=(1-z)^t\sum_{j=0}^{r-1}{\binom{t-1+j}{t-1}z^j}.$$

 Consider the map $\beta(z)=4z(1-z)$ with the dessin $\Delta_\beta$ (see Figure \ref{cleanmap}) and $S_{r,t}(z)$ with the dessin $\Delta_S$ (see Figure \ref{splitstar}). The composition $\beta(z)\circ S_{r,t}(z)$ is a Shabat polynomial that  produces the dessin obtained by coloring the vertices of $\Delta_S$ to black and adding a white vertex of degree $2$ inside every edge (in other words, replacing every edge of $\Delta_S$ with $\Delta_\beta$). Note the number of edges in $S_{r,t}(z)$ is $r+t-1$. The composition produces the new dessin $\Delta_F$ (see Figure \ref{cleansplitstar}) and Shabat polynomial $F(z)=\beta(z)\circ S_{r,t}(z)$ with passport $[r,t,1^{r+t-2};2^{r+t-1}]$, and therefore $F(z)$ equals $$4z^r(1-z)^t \left(\sum_{j=0}^{r-1} {\binom{t-1+j}{t-1}z^j}\right)\left(\sum_{j=0}^{t-1}{\binom{r-1+j}{r-1}}\binom{r+t-1}{r+j}(-1)^jz^j\right),$$ which can be rewritten in terms of hypergeometric functions, as in the statement of the present proposition.
\end{proof}

\begin{figure}
 \begin{center}
    \includegraphics[width=1.5in]{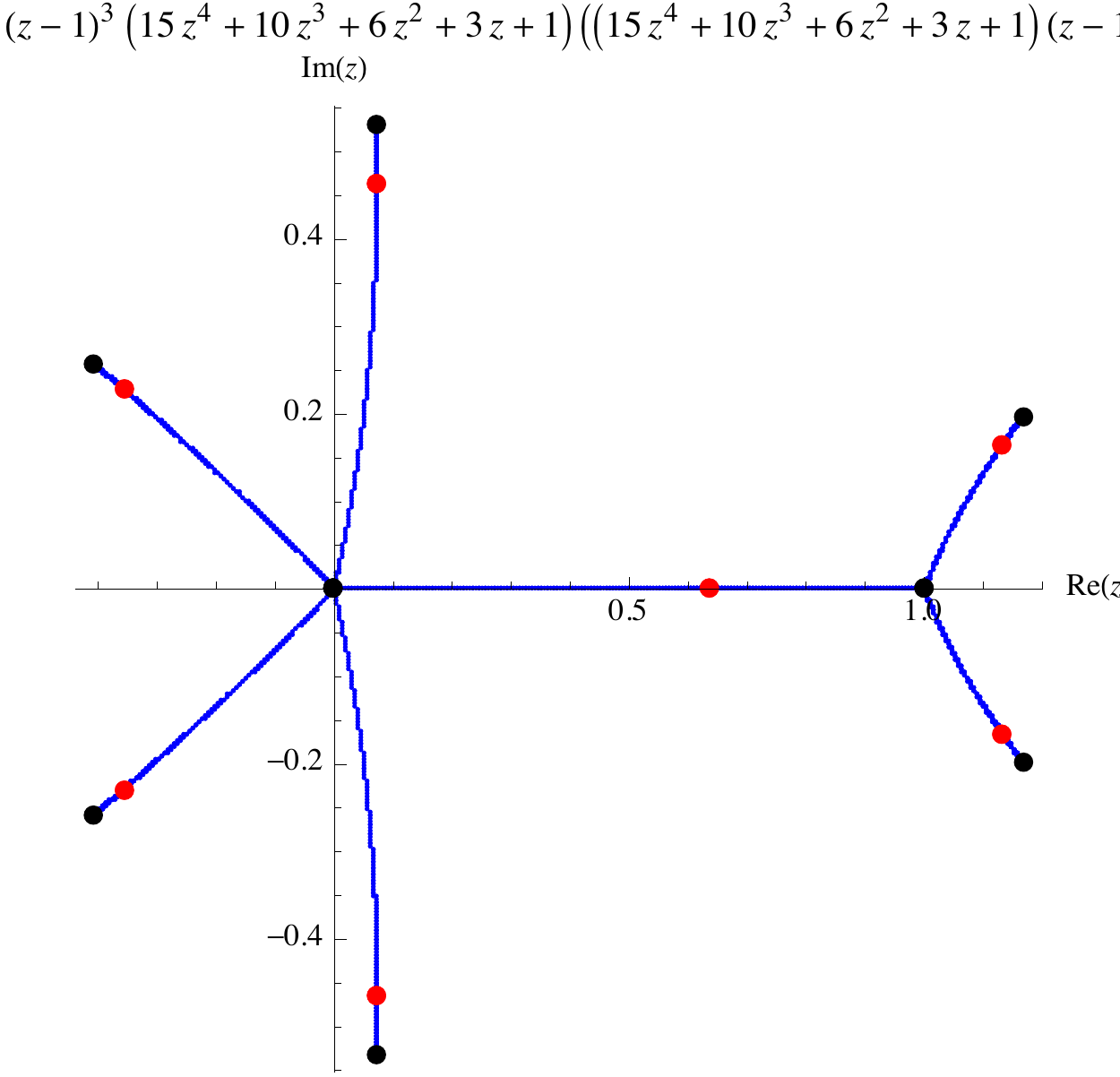}
  \end{center}
  \caption{The tree obtained by the Shabat polynomial in Proposition \ref{cleansplitstarprop} with $r=5, t=3.$}
    \label{cleansplitstar}
  \end{figure}

%%%%%%%%%%%%%%%%%%%%%%%%%%%%%%%%%%%%%%%%%%%%%%%%%%%%%%%%%%

\begin{pro}
The Shabat polynomial for the unique tree having passport $$[r^2,1^{4r-3};3^{2r-1}]$$ with two black vertices of degree $r$ located at $z=0$ and $z=1$ is given by
$$F(z)=(T\circ S_{r})(z),$$
where $T(z)=-(3/2) \sqrt{3} (i+ \sqrt{3} - 2 iz) (z-1) z$ and $$ S_{r}(z)=(1-z)^r\sum_{j=0}^{r-1}{\binom{r-1+j}{r-1}z^j}$$
$F(z)$ is unique up to isomorphism.
\label{classv}
\end{pro}

\begin{figure}
 \begin{center}
    \includegraphics[width=1.5in]{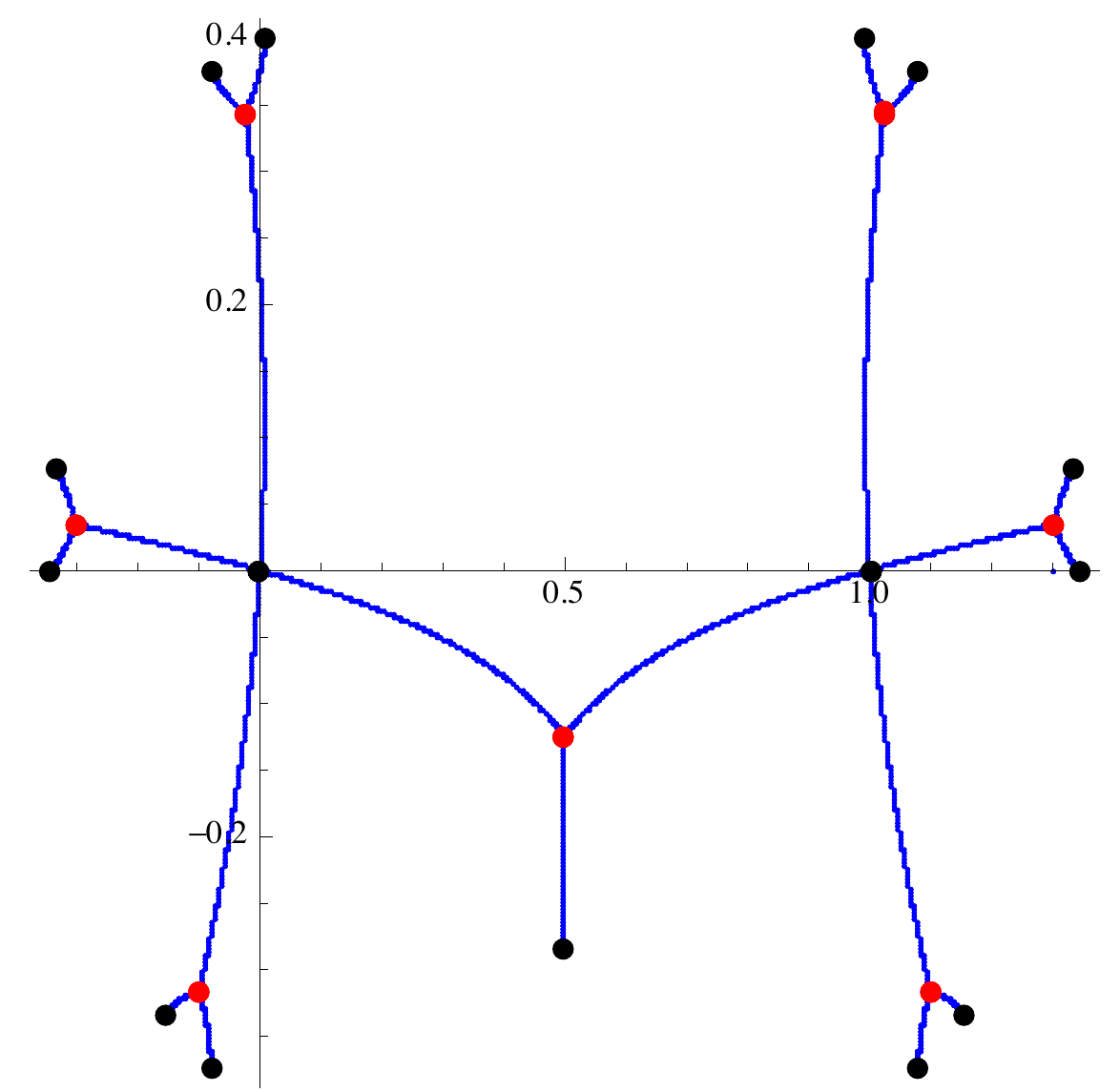}
  \end{center}
  \caption{An illustration of the tree derived from the Shabat polynomial in Proposition \ref{classv} where $r=4.$}
    \label{classvpic}
  \end{figure}

\begin{proof}
First we will show that $T(z):=-(3/2) \sqrt{3} (i+ \sqrt{3} - 2 iz) (z-1) z$ is a $3$-star with a white center and black leaves at $z=0$ and $z=1$. Considering $T(z)$, we see immediately three distinct roots of multiplicity one at $z=0,1,\frac{1-i\sqrt{3}}{2}$ representing three black leaves in $\Delta_F$. Next we consider the derivative of $T(z),$ 

$$T'(z)=9 i \sqrt{3} (1/6 i (3 i + \sqrt{3}) + z)^2$$

\noindent which has a single root of multiplicity $2$. Since the multiplicity of the black vertices is $1$, we may assume that the multiple root in $T'(s)$ must refer to a root of multiplicity $3$ in $F(z)-1,$ representing the white vertex of degree $3$. Therefore, $T(z)$ must be a 3-star with black leaves at $z=1$ and $z=0$. We can now use the idea of composition to replace every edge of the tree having Shabat polynomial $S_{r,r}(z)$ with the 3-star by computing the composition $(T\circ S_{r,r})(z)$ where $S_{r,t}(z)$ is the polynomial defined in the proof of Proposition \ref{cleansplitstarprop}. This will add a white vertex of degree $3$ and an additional black leaf for every edge. Note that $S_{r,r}(z)$ corresponds to a tree with $2r-1$ edges and $4r-2$ vertices. Therefore $\Delta_F$ will have $2r-1$ white vertices of degree $3$ and $4r-3$ black leaves, in addition to the two black vertices of degree $r$.
\end{proof}
%%%%%%%%%%%%%%%%%%%%%%%%%%%%%%%%%%%%%%%%%%%%%%%%%%%%%%%%%%

\begin{pro} \label{cleancat}
 The Shabat polynomial for the tree with passport $[3^3,1^5;2^7]$, a black vertex of degree three at $z=0$ and a black vertex of degree one at $z=1$ is 
 $$F(z)= -\frac{4}{531441}(z-1) z^3 \left(2 z^2+3 z+9\right)^3 \left(8 z^4+28 z^3+126 z^2+189 z+378\right).$$
\end{pro}

\begin{figure}
 \begin{center}
    \includegraphics[width=1.5in]{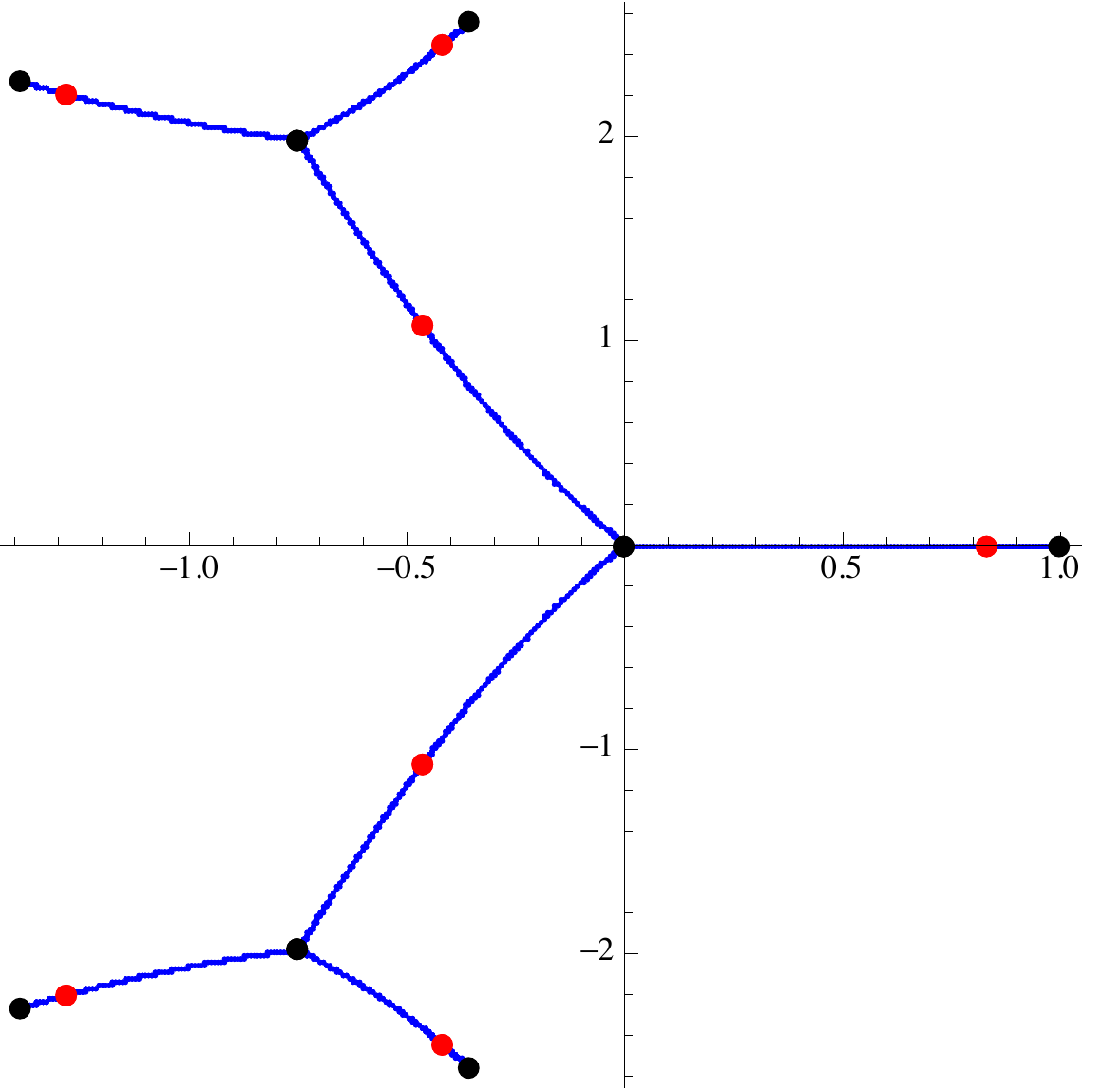}
  \end{center}
  \caption{An illustration of the tree described in Proposition \ref{cleancat}.}
  \label{cleancatfig}
  \end{figure}
  
\begin{proof}
$F(z)=(\beta\circ f)(z)$ where $\beta(z)=4z(1-z)$ and $f(z)= -\frac{1}{729}(z-1)(9+3z+2z^2)^3,$ which is the Shabat polynomial for passport $[3^{2},1;3, 1^4]$ obtained by letting $r=3,s=3,t=1$ in Proposition \ref{iii}.
\end{proof}

%%%%%%%%%%%%%%%%%%%%%%%%%%%%%%%%%%%%%%%%%%%%%%%%%%%%%%%%%%
%%%%%%%%%%%%%%%%%%%%%%%%%%%%%%%%%%%%%%%%%%%%%%%%%%%%%%%%%%
\section{Monodromy Groups for Passports of Size One}\label{monodromygroups}

In this section, we provide proofs for the monodromy groups associated with each passport listed in Table \ref{bigtable}. In all of our proofs, we proceed by choosing a particular labeling of the edges of the dessin. Though the monodromy group does not depend on the choice of labels, some choices better illustrate how $\sigma_0$ and $\sigma_1$ generate the monodromy group.

\begin{pro}
The passports $[r;1^r]$, $[2^r, 1; 2^r, 1]$, $[2^r; 2^{r-1}, 1^2]$ have respective monodromy groups $\mathbb{Z}_r$, $D_{2(2r+1)}$, $D_{2(2r)}$.
\end{pro}

\begin{proof}
The first passport gives the $r$-star dessin with monodromy group generated by the $r$-cycle and the identity permutation.  It follows that the monodromy group is the cyclic group $\mathbb{Z}_r$. The second and third passports yield the path dessins with $2r+1$ and $2r$ edges respectively. We handle these two cases simultaneously, since the argument is essentially the same. The dessins in Figure \ref{pathdessins} are examples of path dessins.

\begin{figure}
\begin{center}
\includegraphics[width=0.4\textwidth]{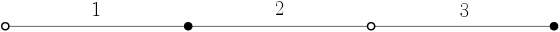} \hspace{0.05\textwidth}
\includegraphics[width=0.4\textwidth]{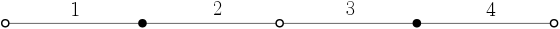}
\end{center}
\caption{The path dessins of 3 and 4 edges, respectively.}
\label{pathdessins}
\end{figure}

 In both cases, the generators of the groups $\sigma_{0}$ and $\sigma_{1}$ have order 2, and the respective $\sigma_{\infty}$'s have order $2r+1$ and $2r$. 
Since in this case $\sigma_{\infty} = (\sigma_{0}\sigma_{1})^{-1} = \sigma_1 \sigma_0$, we may view the monodromy group as $\langle \sigma_{0}$, $\sigma_{\infty} \rangle$. 
We let $n$ denote the order of $\sigma_\infty$; note that $n$ is either $2r +1$ or $2r$ depending on the passport. 
The relations $\sigma_0^{2} = \sigma_\infty^{r} = 1$ and $\sigma_0\sigma_\infty = (\sigma_0 \sigma_1) \sigma_0 = (\sigma_1 \sigma_0)^{-1} \sigma_0 = (\sigma_\infty)^{-1} \sigma_0$ hold. 
The conclusion is that the monodromy groups of these dessins are isomorphic to the dihedral groups of order 2$n$.
\end{proof}

%%%%%%%%%%%%%%%%%%%%%%%%%%%%%%%%%%%%%%%%
%Shabat Polynomial for crosette

\begin{pro}\label{prop:str}
Assume r > 1. The passport $[s^{r-1},t;r, 1^{(r-1)(s-1)+(t-1)}]$ has $n= (r-1)s +t$ edges and a unique tree with monodromy group $G$, where \[ G \cong \begin{cases} 
      \mathbb{Z}_r \wr \mathbb{Z}_s, & \text{if $s = t$} \\ 
      S_\frac{s(r-1) + t}{d} \wr \mathbb{Z}_d, & \text{if $s\neq t$, $\gcd(s,t) =d$, $\gcd(d,\frac{n}{d})=1$, $r$ even} \\
           \widetilde{R}_d & \text{if $s\neq t$, $\gcd(s,t) =d$, $\gcd(d,\frac{n}{d})=1$, $r$ odd and $\frac{t}{d}$ is even} \\
      
      A_\frac{s(r-1) + t}{d} \wr \mathbb{Z}_d, & \text{if $s\neq t$, $\gcd(s,t) = d$, $\gcd(d,\frac{n}{d})=1$, $r$ is odd and $\frac{t}{d}$ is odd} \\
      ??? & otherwise.
   \end{cases}
\]
\label{iv}
The group $\widetilde{R}_d$ denotes the index $2^{d-1}$ subgroup of $S_{\frac{n}{d}} \wr \mathbb{Z}_d$ such that for all $(\tau_1, \dots, \tau_d, g) \in \widetilde{R}_d$, $\sgn(\tau_1)=\sgn(\tau_2)=\dots=\sgn(\tau_d)$.
\end{pro}
\begin{proof}
The passport $[s^{r-1},t;r, 1^{(r-1)(s-1)+(t-1)}]$ produces a tree of diameter four with $n = (r-1)s +t$ edges in the non-degenerate cases. 
\begin{figure}
\begin{center}
\hspace{15mm} \includegraphics[scale = 0.4]{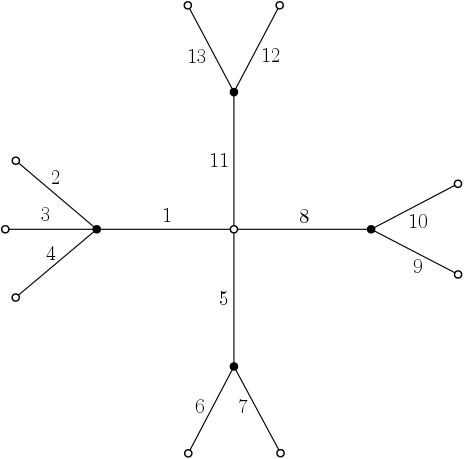}
\newline
\caption{An example of a dessin from Proposition \ref{classiii} where $r=4,s=3,t=4$.}
\end{center}
\end{figure}
\newline
\noindent In general, $\s$ is the product of one $t$-cycle and $(r-1)$ many $s$-cycles and $\sigma_1$ is an $r$-cycle. We label our edges so that we compute the permutations $\s, \ss, \sss$ as
\begin{align*}
\sigma_0 &= (1, \hdots, t)(t+1, \hdots, t+s)(t+s+1, \hdots, t+2s)\hdots(t+(r-2)s+1, \hdots, t+(r-1)s) \\
\sigma_1 &= (1, t+1, t+s+1, t +2s+1, \hdots, t+(r-2)s+1) \\
\sss^{-1} &= \sigma_0\sigma_1 = (1,2,\dots, n)
\end{align*}
(Note that we go left to right when computing permutation products.)
\noindent {\bf Case 1:} $s=t\implies G=\mathbb{Z}_r\wr \mathbb{Z}_s$ \\
Assume $s = t$. Then our dessin is the composition of an $s$-star with an $r$-star, which means $G$ is a subgroup of $\mathbb{Z}_r \wr \mathbb{Z}_t$. Define $\tau_i := \s^{-i} \ss \s^{i}$. Referring to the above where we already computed $\s$ and $\ss$, we see that 
\begin{eqnarray*}
\tau_0 &=& (1,t+1,2t+1, \hdots, (r-1)t + 1)=\sigma_1\\
\tau_1 &=& (2, t+2, 2t+2, \hdots, (r-1)t + 2)\\
& \vdots & \\
\tau_{t-1} &=& (t, 2t, 3t,\hdots, rt)
\end{eqnarray*}
Each $\tau_{i}$ is an $r$-cycle and generates $\mathbb{Z}_r$. Since the $\tau_{i}$'s partition $\{1, 2, \hdots, rt\}$, they must commute with each other and we see that together they generate $\mathbb{Z}_r^t$. Also, $\s$ is a product of $t$-cycles satisfying $\s^{-1} \tau_i \s = \tau_{i+1 }$ where the subscripts are reduced modulo $t$. These relations are sufficient to recognize that $G$ contains $\langle \s, \tau_1, \tau_2, \hdots, \tau_{t-1}\rangle\cong \mathbb{Z}_r \wr \mathbb{Z}_t $.

%These relations are sufficient to recognize the semi-direct product of $\mathbb{Z}_r^t \rtimes \mathbb{Z}_t$ given by $\phi: \s \mapsto (\pi: \mbox{Cycles every entry of a $t$-tuple in $\mathbb{Z}_r^t$ to the right})$, which is precisely the desired wreath product. Thus $\mathbb{Z}_r \wr \mathbb{Z}_t \cong \langle \s, \tau_1, \tau_2, \hdots, \tau_{t-1}\rangle \subseteq G$.
\vspace{5mm}

\noindent {\bf Case 2:} $s\neq t, \gcd(s,t)=1 \implies G=A_n \text{ for } r,t \text{ odd } $ and $G=S_n$ otherwise \\
Assume that $\gcd(s,t) = 1$, with $s$ or $t > 1$. It is known that a permutation group containing $(1,2,3)$ and $(1,2, \hdots, n)$ contains an isomorphic copy of $A_{n}$; a proof can be seen in \cite{conrad_generating_2013}. Our goal is to show that $A_n \le G \le S_n$ and then use a parity argument to determine which containment is improper. Given that $\sigma_0\sigma_1=(1,2,\dots,n)\in G$ we proceed to show $(1,2,3)\in G$.
%We already know $G \le S_n$, so we seek the other inequality. To this end, we make the computation  $\sigma_0\sigma_1$ = $(1,2,3,\hdots, n)$. Recovering $(1,2,3)$ is slightly more difficult. 
%There are some subtleties arising from cases of small $t$ that must be handled separately, but in all instances we can have $s$ and $r$ be arbitrary, so long as $\gcd(s,t) = 1$ and $s$ or $t > 1$. 
\newline
\newline
Assume $t=1$ and $s>1$. We claim $\rho:=(\sigma_0^{-1}\sigma_1^{-1}\sigma_0)(\sigma_\infty\sigma_1\sigma_\infty^{-1})=(1,2,3)$.
%We have $(1,2,3)$ = $\sigma_0^{-1} \sigma_1^{-1} \sigma_0 \sigma_1^{-1} \sigma_0^{-1} \sigma_1 \sigma_0 \sigma_1$.Recall that we defined the $n-$cycle $\sigma_{\infty}$ to be the permutation satisfying $\sigma_0 \sigma_1 \sigma_{\infty}=1$, therefore the previous expression can be simplified as follows:
%\begin{equation*}
%\sigma_0^{-1} \sigma_1^{-1} \sigma_0 \sigma_1^{-1} \sigma_0^{-1} \sigma_1 \sigma_0 \sigma_1 = \underbrace{\sigma_0^{-1} \sigma_1^{-1}\sigma_0} \underbrace{\sigma_{\infty} \sigma_1 \sigma_{\infty}^{-1}}
%\end{equation*}and 
%Notice that this expression is a product of the conjugations of $\sigma_1^{-1}$ and $\sigma_1$ by $\sigma_0$ and $\sigma_{\infty}^{-1}$ respectively (emphasized in braces). 
Since $t=1$, $\s$ is a product of $(r-1)$ $s$-cycles while $\ss, \ss^{-1}$ remain $r$-cycles. We see that \begin{equation*}\rho=(\sigma_0^{-1} \sigma_1^{-1}\sigma_0)(\sigma_{\infty} \sigma_1 \sigma_{\infty}^{-1})  =  (1,(r-2)s+3,\ldots, 2s+3, s+3,3)(2,3,s+3,2s+3,\ldots,(r-2)s+3)\end{equation*}
%These are two non-disjoint $r$-cycles. 
One may verify that $\rho(1)=2,\rho(2)=3,\rho(3)=1$ and, for $k>3$, $\rho(k)=k$. It follows that $A_n \leq G$.
%\begin{eqnarray*}
%(\s^{-1}\ss^{-1}\s)(\sss \ss \sss^{-1}) (1) & = & (\sss \ss \sss^{-1})((r-2)s+3) \\
%& = & 2\\
%(\s^{-1}\ss^{-1}\s)(\sss \ss \sss^{-1}) (2)& = & (\sss \ss \sss^{-1})(2) \\ 
%& = & 3 \\
%(\s^{-1}\ss^{-1}\s)(\sss \ss \sss^{-1}) (3) & = & (\sss \ss \sss^{-1})(1) \\
%& = & 1  \quad \mbox{as desired.}
%\end{eqnarray*}
%Now let $3<k\leq (r-1)s+1$ (in this case $(r-1)s+1$ is the number of edges) and notice that either both $(\s^{-1}\ss^{-1}\s$ and $\sss \ss \sss^{-1})$ fix $k$ or the $\s^{-1}\ss^{-1}\s$ acts on $k$ and $\sss \ss \sss^{-1}$ reverses this action; these cases depend on $s,r$ and $k$ but the dependence does not affect the fact that the permutation $(\s^{-1}\ss^{-1}\s \sss \ss \sss^{-1})$ is indeed the three cycle $(1,2,3)\in G(s,t,r)$. 
\newline
\newline
If $t = 2$, we have $\sigma_0^s=(1,2)\in G$. Since $G$ contains the transposition $(1,2)$ and the cycle $(1,2,\hdots,n)$, then $ S_n \leq G$.
\newline
\newline
Now suppose $t \geq 3$, we first set $k$ to be the smallest positive integer such that $k$ satisfies $k \equiv 0$ (mod $s$) and $k \equiv -1$ (mod $t$). The existence of such a number is guaranteed by the Chinese Remainder Theorem. We claim $\rho:= (\sigma_1^{-1} \sigma_0^k \sigma_1) \sigma_0^k (\sigma_1^{-1} \sigma_0^{-2k} \sigma_1)=(1,2,3) $. %To see this, we compute the necessary ingredients and build up to our permutation
Notice that
\begin{eqnarray*}
%\s^{k} &=&(1, t, \hdots, 2) \\
%\ss^{-1} \s^{k} \ss &=& (t+1, t, \hdots, 2) \\
%\ss^{-1} \s^{-2k} \ss &=& (t+1, 2, \hdots, t)^2 \\
(\sigma_1^{-1} \sigma_0^k \sigma_1) \sigma_0^k (\sigma_1^{-1} \sigma_0^{-2k} \sigma_1) &=& (t+1,t,\hdots, 3,2)(1, t, \hdots,3, 2)(t+1,2,3,\hdots,t)^2.
\end{eqnarray*}

\noindent One may verify that $\rho(1)=2, \rho(2)=3,\rho(3)=1$ and $\rho(k)=k$ for $k>3$. Thus $\rho=(1,2,3) \in G$ and therefore $A_n \subseteq G$.
%\noindent It is clear now that the cycle (1,2,3) is part of this permutation. To show that it is the entire permutation, we let $m$ be an integer such that $3 < m \le (r-1)s + t$. When $m > t + 1$, it is fixed by $\sigma_1^{-1} \sigma_0^k \sigma_1 \sigma_0^k \sigma_1^{-1} \sigma_0^{-2k} \sigma_1$. When $m \le t+1$, we have $m \mapsto (m-1) \mapsto (m-2) \mapsto (m-1) \mapsto m$ Thus $(1,2,3) \in G(s,t,r)$ and the conclusion is $A_n \subseteq G(s,t,r)$.
\newline
\newline
For every triple $s,t$ such that $\gcd(s,t) =1$ and $s$ or $t > 1$, we have shown that $A_{n} \subseteq G$. Since we also have $G\leq S_{n}$, by index considerations $G$ is either the symmetric or alternating group of appropriate order. Otherwise if $r$ or $t$ is even, $\sigma_0$, being the product of a $t$-cycle and $(r-1)$ $s$-cycles, is an odd permutation (note $s$ must be odd if $t$ is even), so $G \cong S_n$. Since both $\sigma_0$ and $\sigma_1$ are even permutations when $r$ and $t$ are odd, we deduce that $G \leq A_n$ and thus the double inclusion gives us $G \cong A_n$.

\vspace{5mm}

\noindent {\bf Case 3:} $s\neq t, \gcd(s,t)=d>1,\gcd(d,\frac{n}{d})=1 \implies G=A_{\frac{n}{d}}\wr \mathbb{Z}_d \text{ for } r,\frac{t}{d} \text{ odd }$, $G=S_{\frac{n}{d}}\wr \mathbb{Z}_d$ for $r$ even, and $\widetilde{R}_d$ for $r$ odd, $\frac{t}{d}$ even.\\

In this final case, we assume $\gcd(s,t) = d > 1$. This tree is the composition $P\circ Q$ where $P$ is the $d$-star and $Q$ is the dessin corresponding to the passport $\left[\left(\frac{s}{d}\right)^{r-1},\frac{t}{d}; r, 1^{(r-1)(s/d-1) + (t/d-1)}\right]$. Hence, the monodromy group $G$ is a subgroup of the wreath product $G_Q \wr \Z_d$, where $G_Q$ is the monodromy group for $Q$. %To obtain the reverse inclusion, we show that with the generators $\sigma_0$, $\sigma_1$, we can extract an isomorphic copy of $G_Q \wr \Z_d.$ For the rest of this proof, by "the wreath product", we will be referring to $G_Q \wr \Z_d$. 
%Do something else to show other inclusion, i.e. krasner/kaloujnine or reproducing sigma0/sigma1.

 Consider the partition of $\{1, \dots, n\}$ into the $d$ sets $\{1,d+1,\dots, n-d+1\}, \{2,d+2,\dots,n-d+2\}, \dots, \{d,2d,\dots,n\}$, each of size $\frac{n}{d},$ and denote them $P_1, \dots, P_d$ respectively. Recall that $\sigma_0$ is the disjoint product of a $t$-cycle and $(r-1)$ $s$-cycles, and moreover every element in $\{1,2,\dots, n\}$ is moved by exactly one of these cycles under the canonical group action. Because $d$ divides both $s$ and $t$, $\tau := \sigma_0^d$ is the disjoint product of $d$ $\frac{t}{d}$-cycles and $d(r-1)$ $\tfrac{s}{d}$-cycles. Moreover, each disjoint cycle of $\tau$ permutes elements in exactly one of the $P_i$ while fixing the rest. Similarly, because $d$ divides $n$, $\sigma_\infty^d$ is the disjoint product of $d$ $\tfrac{n}{d}$-cycles, and each disjoint cycle of $\sigma_\infty^d$ likewise permutes elements in exactly one of the $P_i$. Note that $\sigma_1$ permutes only the elements of $P_1$.

Let $k$ be the smallest positive integer such that $k$ satisfies $k \equiv 0$ (mod $\frac{s}{d}$) and $k \equiv -1$ (mod $\frac{t}{d}$). One may verify that $\rho := \ss^{-1}\tau^k\ss\tau^k\ss^{-1}\tau^{-2k}\ss =  (1,d+1,2d+1)$. (Note that in the case where $t=d$, we let $\rho:= (\tau^{-1}\sigma_1^{-1}\tau)(\sigma_\infty^d\sigma_1\sigma_\infty^{-d})$ and proceed with the same argument.)

%When its action is restricted to the set $P_1$, we are left with the previous case where $\gcd(\frac{s}{d},\frac{t}{d})=1$; by relabeling $\{1, d+1, \dots, n-d+1\}$ to $\{1,2,\dots,\frac{n}{d}\}$ and applying the argument used above we find that $\rho\large{|}_{P_1} = (1,d+1,2d+1)$. For $i>1$ and $x \in P_i$, $\ss$ fixes $x$ and thus commutes with $\tau^k$, allowing for cancellation. Thus, we see that $\rho$ fixes $x$. Together with the above, it follows that $\rho = (1,d+1,2d+1)$. 
%Conjugating by successive powers of $\sigma_\infty^d$ as in Case 2 gives the rest of the generators of the alternating group; together with $\ss$, these will then generate the alternating/symmetric group, depending on the parity of $r$ and $\tfrac{t}{d}$. 

Now assume $\gcd\left(d,\frac{n}{d}\right)=1.$  We can conclude that the subgroup $$N = \langle \rho, \sigma_\infty^{-d}\rho\sigma_\infty^d, \sigma_\infty^{-2d}\rho\sigma_\infty^{2d}, \dots, \sigma_\infty^{-(n-d)}\rho\sigma_\infty^{n-d},\ss \rangle$$ is isomorphic to $S_{\frac{n}{d}}$ when $r$ is even and isomorphic to $A_{\frac{n}{d}}$ when $r$ is odd \cite[Theorem 3.4]{conrad_generating_2013}.  Observe that $\psi := (\sss)^{\frac{n}{d}}$, the product of $\tfrac{n}{d}$ disjoint $d$-cycles, will take elements of $P_i$ to $P_{i+1}$, where the addition is modulo $d$ (with $P_0 = P_d$). Similarly, the $j$th power of $\psi$ will take elements of $P_i$ to $P_{i+j}$, where again the addition is modulo $d$. Thus, conjugating $N$ by $\psi^j$ will produce an isomorphic copy of $N$, which instead acts on the set $P_{1+j}$ which is the relation we expect for the wreath product. 

Thus, we have shown that $G\cong S_{\frac{n}{d}}\wr \Z_d$ when $r$ is even and $G\ge A_{\frac{n}{d}}\wr \Z_d$ when $r$ is odd. Henceforth, assume that $r$ is odd. When $r$ is odd, $\sigma_1$ is an element of $A_{\frac{n}{d}}\wr \Z_d$. Therefore, we must examine $\sigma_0$ or $\sigma_\infty$ to determine $G$.

Recall that $\psi$ is the element of the wreath product which acts on the set $\{P_1,\dots,P_d\}$. That is, $\psi$ takes all the elements of $P_i$ to $P_j$ for some $j$. Thus by multiplying and element $g\in G$ by an appropriate power of $\psi$, we may assume the element $g$ sits inside of $S_{\frac{n}{d}}^d$. Observe that $\widetilde{G}=(S_{\frac{n}{d}}^d)\slash (A_{\frac{n}{d}}^d)$ is isomorphic to $d$ copies of $\Z_2$. Consider $\sigma_\infty^{d}$. One can check that $G=\langle\sigma_1,\sigma_\infty^{d},\sigma_\infty^{\frac{n}{d}} \rangle$. Observe that $\sigma_1=(0,\dots,0)\in\widetilde{G}$ since $\sigma_1\in A_{\frac{n}{d}}^d$. If $\frac{t}{d}$ is odd then $\frac{n}{d}$ is odd since $r-1$ is even. Therefore, $\sigma_\infty^{d},\sigma_\infty^{\frac{n}{d}}\in A_{\frac{n}{d}}^d$ and hence $G\cong A_{\frac{n}{d}}\wr \Z_d$. If $\frac{t}{d}$ is even then $\frac{n}{d}$ is even since $r-1$ is even. Therefore, $\sigma_\infty^{d}$ and $\sigma_\infty^{\frac{n}{d}}$ are isomorphic to $(1,1,\dots,1)$ in the quotient group $\widetilde{G}$ and thus $G=\widetilde{R}_d$.
%\bigskip

%[WE NEED TO DECIDE IF WE WANT TO FIGURE OUT WHAT HAPPENS OTHERWISE, THAT IS WHEN GCD(D,N/D) IS NOT 1 OR JUST LEAVE THIS AS AN OPEN QUESTION.]
%It follows that the subgroups $N, \psi^{-1}N\psi, \psi^{-2}N\psi^2, \dots, \psi^{-(d-1)}N\psi^{d-1}$ are each isomorphic copies of $G(r,\tfrac{s}{d},\tfrac{t}{d})$, acting on the subsets $P_1, \dots, P_d$ respectively. From the argument above, we can conclude that for $x_i$ a generator of the copy of $G(r,\tfrac{s}{d},\tfrac{t}{d})$ which acts on $P_i$, and $x_{i+1}$ its corresponding generator in the next copy of $G(r,\tfrac{s}{d},\tfrac{t}{d})$, $\psi x_i = (\psi x_i \psi^{-1}) \psi = x_{i+1}\psi$, which is the relation we expect for the wreath product.
\end{proof}

%%%%%%%%%%%%%%%%%%%%%%%%%%%%%%%%%%%%%%%
%Shabat Polynomial for clean double star

\begin{lemma}\label{kcleaning}
Suppose that $\pi_0, \pi_1 \in S_n$ with $\langle \pi_0, \pi_1 \rangle \ge A_n$ with $n\ge 5$.  
\begin{enumerate}
\item If  $|\pi_0| \not = |\pi_1|$, then $\Gamma=\langle (\pi_0,\pi_1),(\pi_1,\pi_0) \rangle$ must contain $A_n\times A_n$.   
\item $\Gamma=\langle (\pi_0,\pi_1, \id), (\id, \pi_0, \pi_1), (\pi_1,\id,\pi_0) \rangle $ must contain $A_n\times A_n \times A_n$.
\end{enumerate}
\end{lemma}

\begin{proof}
Suppose that $\id\neq\rho \in A_n$. Observe that $\langle \tau^{-1}\rho\tau: \tau\in A_n\rangle$ is a normal subgroup of $A_n.$  If $n\geq 5$, $A_n$ is simple and therefore, $A_n=\langle \tau^{-1}\rho\tau: \tau\in A_n\rangle$.  

First, we consider Statement 1. Suppose that $(\rho,\id)\in \Gamma$.  We want to show that $A_n\times \langle \id \rangle$ is a subgroup of $\Gamma$.  There is a homomorphism $ \mathrm{proj}: S_n\times S_n \to S_n$, which is a projection from the first component.  Since $A_n \le \langle \pi_0,\pi_1\rangle$, the image of $\mathrm{proj}(\Gamma) \ge A_n.$  Therefore, for all $\tau\in A_n$ there exists $\tau'\in S_n$ such that $(\tau, \tau')\in \Gamma$. Conjugating $(\rho,\id)$ by all $(\tau, \tau')$ shows that $A_n\times \langle \id \rangle\le \Gamma$.  Note that the same argument can be used to show $\langle \id \rangle \times A_n\le \Gamma$ via projection in the other component.  Statement 1 then follows as long as $\rho\neq \id$ exists.  Furthermore, the argument to establish statement 2 would proceed in an identical fashion, presuming $\rho\neq \id$ exists.

To establish existence of $\rho$ in the case of Statement 1, we claim that there exists an element of the form $(\rho,\id) \in \Gamma$ where $\rho \not = \id$. Without loss of generality, assume $|\pi_0| > |\pi_1| $, and then consider $(\pi_0,\pi_1)^{|\pi_1|}, (\pi_1,\pi_0)^{|\pi_1|}$, in which case we may let $\rho=\pi_0^{|\pi_1|}$.

Now we prove such an element exists in the case of statement 2 for $n>2$. If $|\sigma_0|\ne |\sigma_1|$, then the proof is analagous to the argument for statement 1. Otherwise $|\sigma_0| = |\sigma_1|=r$ and we want to find some element $\pi \in A_n$ such that $|\pi| \nmid r$. One can show that such a $\pi$ exists by proving that, for $n>2$, there must be some prime $q$ \emph{not} dividing $|\pi_0|=r$. One can show $q$ exists by using the fact that
$$
n<\displaystyle\sum_{\substack{p\le n \\ p \textup{ prime}}}{p}
$$
for $n>2$. Using all three generators of $\Gamma$, one can produce the element $(\pi_0^{k_1},\pi,\pi_1^{k_2})\in A_n^3$ where $k_1,k_2\in\mathbb{Z}$. By raising this element to the $r^{th}$ power, we produce the element $(\id,\pi^r,\id)\in \Gamma$ and let $\rho=\pi^r$.
\end{proof}

\begin{corollary}
Let $H$ be a simple group. Suppose that $\pi_0, \pi_1 \in S_n$ with $\langle \pi_0, \pi_1 \rangle \ge H$.  
\begin{enumerate}
\item If  $|\pi_0| \not = |\pi_1|$, then $\Gamma=\langle (\pi_0,\pi_1),(\pi_1,\pi_0) \rangle$ must contain $H\times H$.   
\item $\Gamma=\langle (\pi_0,\pi_1, \id), (\id, \pi_0, \pi_1), (\pi_1,\id,\pi_0) \rangle $ must contain $H\times H \times H$.
\end{enumerate}
\end{corollary}

\begin{remark}
In \cite{adrianov_primitive_1997}, Adrianov, Kochetkov, and Suvorov classify all the possible primitive, and thus simple, monodromy groups of plane trees.
\end{remark}

%\noindent For the remaining passports, recall that $R_m$ is the index 2 subgroup of $S_{\frac{n}{m}} \wr \mathbb{Z}_m$ such that for all $(\tau_1, \dots, \tau_m, g) \in R_m$, $\tau_1 \tau_2\cdots\tau_m$ is an even permutation and that here $n$ refers to the number of edges of the dessin.
\begin{pro}
Let $r,t>1.$  The passport $[r,t,1^{r+t-2};2^{r+t-1}]$ produces a unique tree with monodromy group $G$, where \[ G \cong \begin{cases} 
      A_{2r-1} \times \mathbb{Z}_2 & r=t, r \text{ odd}\\
      S_{2r-1} \times \mathbb{Z}_2 & r=t, r \text{ even}\\
      A_{r+t-1} \wr \mathbb{Z}_2 & r\not=t, \text{ both odd} \\
      R_2 & r\not=t, \text{ both even} \\
      S_{r+t-1} \wr \mathbb{Z}_2 & r\not=t, \text{ else} \\

   \end{cases}
\]
where $R_2$ denotes the index $2$ subgroup of $S_{r+t-1} \wr \mathbb{Z}_2$ containing $A_{r+t-1} \wr \mathbb{Z}_2$ and an element $(\zeta_1,\zeta_2,0)$ where $\sgn(\zeta_1)=\sgn(\zeta_2)=-1$.
\end{pro}

\begin{proof}
First, we note that this dessin is the composition $P\circ Q$, where $P$ is the 2-star and $Q$ is the dessin of Proposition \ref{iv} with $s = 1$.
\begin{figure}
\begin{flushleft}
\hspace{16mm}\includegraphics[width=0.2\textwidth]{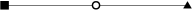}\\
\end{flushleft}

\begin{center}
\includegraphics[width=0.4\textwidth]{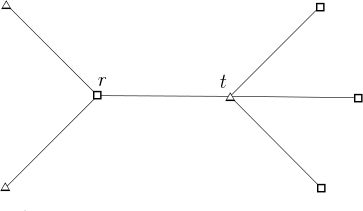}
\hspace{0.05\textwidth} \includegraphics[width=0.47\textwidth]{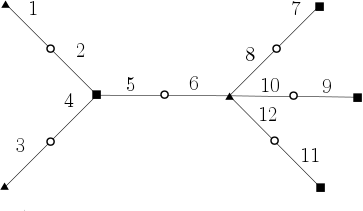}
\end{center}
\vspace{-5mm}
\caption{$P$ and $Q$ on the left, $P\star Q$ on the right}
\end{figure}

Let $G_Q=\langle (1,2,\ldots,r),(r,r+1,\ldots,r+t-1)\rangle $ be the monodromy group of $Q$. By Proposition \ref{iv}, we know that
\[G_Q \cong \begin{cases}
A_{r+t-1} & r,t \text{ both odd} \\
S_{r+t-1} & \text{ otherwise}
\end{cases}\]

\noindent The dessin with passport $[r,t,1^{r+t-2};2^{2r-1}]$ is the composition of $P$ and $Q$, and so its monodromy group $G \le G_Q \wr \mathbb{Z}_2$ \cite{adrianov_composition_1998}. We consider $G$ in two cases: for $r=t$ and $r\not=t$.
\vspace{5mm}

\noindent {\bf Case 1} \\
%We should save the most general case for Paper 2.  
In the first case, we have $r\not=t$. We label our edges in such a way that $\sigma_0$ and $\sigma_1$ are as follows:
\begin{align*}
\sigma_0 &= (1, 2, \hdots, r)(\overline{r}, \overline{r+1}, \hdots, \overline{r+t-1}) \\
\sigma_1 &= (1,\overline{1})(2,\overline{2})\dots(r+t-1,\overline{r+t-1})
\end{align*}
%We show a result which implies that $A_{r+t-1}^2 \leG$, then use an index-counting argument to determine which subgroup $A_{r+t-1}\wr\mathbb{Z}_2 \le G \le S_{r+t-1}\wr\mathbb{Z}_2$ is correct.

\noindent Note that $\s$ is the disjoint product of an $r$-cycle with a $t$-cycle; call these cycles $\pi_1$ and $\pi_2$ respectively. Consider the embedding $\phi: G \rightarrow S_{r+t-1} \wr \mathbb{Z}_2$ given by 
\begin{eqnarray*}
\s & \longmapsto & (\pi_1, \pi_2, 0) \\
\ss & \longmapsto & (\id, \id, 1) 
\end{eqnarray*}
Note that $\ss^{-1} \s \ss$ is mapped to $(\pi_2,\pi_1, 0)$. Apply Lemma \ref{kcleaning} to $n = r+t-1$ (assume $n \ge 5$ for now), $\pi_1, \pi_2 \in S_{r+t-1}$. We have $G_Q=\langle \pi_1, \pi_2 \rangle \ge A_n$ as noted above. Lemma \ref{kcleaning} implies that $A_{r+t-1} \wr \mathbb{Z}_2 \le \phi(G) \le S_{r+t-1} \wr \mathbb{Z}_2$. When $r,t$ are odd, both $\pi_1$ and $\pi_2$ are even permutations, and we see $\phi(G)\cong A_{r+t-1} \wr \mathbb{Z}_2$. When $r$ and $t$ have different parity, we know that $\langle \pi_1, \pi_2 \rangle \cong S_{r+t-1}$, so then $\phi(G) \cong S_{r+t-1} \wr \mathbb{Z}_2$. When $r,t$ are both even, for any $(\rho_1,\rho_2,g) \in \phi(G)$, $\rho_1$ and $\rho_2$ will share the same parity. Since we can take $\rho_1 = \pi_1$, an odd permutation, we see that $\phi(G)$ is properly contained in between $A_{r+t-1} \wr \mathbb{Z}_2$ and $S_{r+t-1} \wr \mathbb{Z}_2$. It is in fact the group $R_2$ described earlier after Table 1. In the finite number of cases where $r+t-1 < 5$, one can verify the result by hand.

%We now need to handle the few cases where $r+t-1 < 5$. Since the dessin is symmetric in that we may flip $r$ and $t$ and still get the same monodromy group, we assume without loss of generality that $r < t$. There are four cases: $(r,t) \in \{(1,2),(1,3),(1,4),(2,3)\}$.    

\vspace{5mm}
\noindent {\bf Case 2} \\
In the second case, we consider $r = t$. We can label our dessin in such a way that:
\begin{align*}
\sigma_0 &= (1, 2, \hdots, r)(\overline{1}, \overline{2}, \hdots, \overline{r}) \\
\sigma_1 &= (1,\overline{r+1})(2,\overline{r+2})\dots(r-1,\overline{2r-1})(r,\overline{r})(r+1,\overline{1})\dots (2r-1,\overline{r-1})
\end{align*}

%We first observe that the wreath product $S_{2r-1} \wr \mathbb{Z}_2$ embeds naturally into the symmetric group $S_{4r-2}$. The two $S_{2r-1}$ copies go into two disjoint sets of $(2r-1)$ elements, say $\{1,2,\hdots, 2r-1\}$ and $\{\overline{1}, \overline{2}, \hdots, \overline{2r-1}\}$ and the "wreath" goes to a product of transposition pairing up elements from each set (for example, $(1, \overline{1})(2,\overline{2})\hdots(2r-1,\overline{2r-1}))$. We can follow this embedding on $\sigma_0$ and $\sigma_1$ to find $G_Q$ as a subgroup of $S_{2r-1} \wr \mathbb{Z}_2$. As suggested earlier, we consider $\{1,2,\dots,2r-1\}$ and $\{\overline{1},\dots,\overline{2r-1}\}$ to be the partitions of $\{1,\overline{1},\dots,2r-1,\overline{2r-1}\}$ on which the respective copies of $S_{2r-1}$ act. We then make the computations:
\noindent Observe that
\begin{align*}
\sigma_\infty^{(2r-1)} &= (1,\overline{1})  \dots (2r-1,\overline{2r-1}), \\
\tau_1=\sigma_1\sigma_0\sigma_1^{-1}& =(r,r+1,\dots, 2r-1)(\overline{r},\overline{r+1},\dots,\overline{2r-1}),\\
\tau_2=\sigma_\infty^{(2r-1)}\sigma_1 &= (\overline{1},\overline{r+1})(1,r+1)(\overline{2},\overline{r+2})(2,r+2)\dots (\overline{r-1},\overline{2r-1})(r-1,2r-1)(r)(\overline{r}),
\end{align*}
and $G_Q=\langle\sigma_{\infty}^{(2r-1)},\tau_1,\tau_2\rangle$ is a subgroup of $S_{2r-1}\times \Z_2$. Furthermore,
$$
\tau_3=\tau_2 \sigma_1\sigma_0\sigma_1^{-1} \tau_2^{-1}=(1,2,\dots,r)(\overline{1},\overline{2},\dots,\overline{r}).
$$
By Proposition \ref{prop:str}, we see that $\langle\tau_1,\tau_3\rangle$ is $S_{2r-1}$ if $r$ even and $A_{2r-1}$ if $r$ odd, and thus we have our result.

\iffalse
This suggests that:
\begin{align*}
\sigma_\infty^{-(2r-1)} &\longmapsto (\id\times\id,1) \\
\sigma_0 &\longmapsto (\sigma\times\sigma,0) \\
\sigma_1 &\longmapsto (\tau\times\tau,1) 
\end{align*}
where $\sigma = (1,2,\dots,r)$ and $\tau = (1,r+1)(2,r+2)\cdots (r-1,2r-1) (r)$. It follows that the group $G$ generated by $\sigma_0$ and $\sigma_1$ lies in the diagonal of the wreath product. Moreover, because $\tau^{-1}\sigma\tau = (r,r+1,\dots,2r-1)$ is the other generator of $G_Q$, $G$ must contain all of $G_Q$. Hence $G \cong \{(\sigma \times \sigma,g)\in G\wr \mathbb{Z}_2 : \sigma \in G_Q\} \cong G_Q \times \mathbb{Z}_2$.
\fi
\end{proof}

%%%%%%%%%%%%%%%%%%%%%%%%%%%%%%%%%%%%%%
%Belyi map for n^2 and sporadic case

\begin{pro}
The passport $[r^2,1^{4r-3};3^{2r-1}]$ produces a unique tree with monodromy group $G$, where \[ G \cong \begin{cases} 
      A_{2r-1} \wr \mathbb{Z}_3 & r \text{ odd} \\
      R_3 & r \text{ even} \\
\end{cases}
\]
\end{pro}
\begin{proof}
The procedure here is similar to the proof for the previous proposition. We observe that this dessin is the composition $P \circ Q$ where $P$ is the 3-star with passport $[1^3;3]$ and $Q$ is the dessin from Proposition \ref{iv} where $s=1, r=t.$ 

\begin{figure}
\centering
\subfloat[$P$, with vertices marked]{\includegraphics[width=0.25\linewidth]{comp_P-dessin.png}}
\hspace{0.5cm}
\subfloat[$Q$, with vertices marked]{\includegraphics[width=0.25\linewidth]{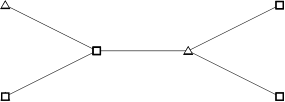}}
\end{figure}

\begin{figure}
\begin{center}
{\includegraphics[width=0.37\textwidth]{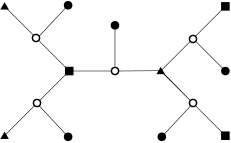}}
\caption{An example of the composition for $r = 3$.}
\end{center}
\end{figure}
\noindent Similarly to the previous passport, we can label the dessin so that
\begin{align*}
\sigma_0 &= (1, 2, \hdots, r)(\overline{r},\overline{r+1}, \hdots, \overline{2r-1})\\
\sigma_1 &= (1, \overline{1}, \widehat{1})(2, \overline{2}, \widehat{2})\hdots(2r-1, \overline{2r-1}, \widehat{2r-1})
\end{align*}
%\begin{lem}
%Suppose that $\s, \ss \in S_n$, for $n \geq 5$ with $\langle \s,\ss \rangle \geq A_n$. Define $x_1 = (\s,\ss, \id)$, $x_2 = (\id, \s, \ss)$, $x_3 = (\ss, \id, \s)$ all in $S_n^3$, and $G = \langle x_1, x_2, x_3 \rangle$. Then $A_n^3 \leq G$. 
%\end{lem}
%\begin{proof} 
%First, we claim that there exists an element of the form $\rho\times\id\times\id \in G$ where $\rho \not = \id$. Let $|\sigma_0| = r$ and $|\sigma_1| = t$. If $r \not = t$, then $x_1^k = (\sigma_0\times\sigma_1\times\id)^t = \sigma_0^t\times\id\times\id$, and take $\rho = \sigma_0^t$. Otherwise suppose $r=t$. There is some other element $\pi \in A_n$ such $|\pi| \not = r$. Because $\langle \sigma_0, \sigma_1 \rangle \ge A_n$, there is some word $\sigma_0^{k_1} \sigma_1^{k_2} \cdots \sigma_1^{k_m} = \pi$. Then $x_1^{k_1} x_3^{k_2} \cdots x_1^{k_m} = (\pi\times\sigma_0^{k_1 + \cdots + k_{m-1}}\times \sigma_1^{k_2 + \cdots + k_m})$ which we raise to the $r$th power to obtain $\pi^r\times\id\times\id$, and we can take $\rho = \pi^r$. The proof of Lemma 1 now holds with $\rho$ in place of $\s^{t}$.
%\end{proof}
\noindent Note that $\s$ is the product of two $r$-cycles (call them $\pi_1$ and $\pi_2$ respectively) and that $\ss$ is the product of $(2r-1)$ $3$-cycles. Consider the embedding $\phi:G \rightarrow S_{2r-1} \wr \Z_3$ defined by
\begin{eqnarray*}
\s & \longmapsto & (\pi_1, \pi_2, \id, 0) \\
\ss & \longmapsto & (\id,\id,\id,1)
\end{eqnarray*}
Under this homomorphism, successive conjugations of $\s$ by $\ss$ are mapped to $(\id,\pi_1,\pi_2,0)$ and $(\pi_2, \id, \pi_1,0)$. Applying Lemma \ref{kcleaning} to $\pi_1,\pi_2$, and $\phi(G)$, we have $A_{2r-1} \wr \Z_3 \le \phi(G) $. When $r$ is odd, both $\pi_1$ and $\pi_2$ are even permutations, so $A_{2r-1} \wr \Z_3 \ge \phi(G)$, giving a double inclusion. When $r$ is even, we consider the quotient group $S_{2r-1}\wr \Z_3\slash A_{2r-1}\wr\Z_3\cong\Z_2\times\Z_2\times\Z_2$. Observe that when $r$ is even $\phi(G)\le R_3$ and $(\pi_1,\pi_2,\id,0)$ is equal to $(1,1,0)$ in the quotient group $\phi(G)\slash A_{2r-1}\wr\Z_3$. We similarly have $(0,1,1)$ and $(1,0,1)$ in the quotient group. Hence, we see that $\phi(G)$ is an index two subgroup of $S_{2r-1}\wr \Z_3$ and thus $\phi(G)\ge R_3.$
\end{proof}

\begin{pro}
The passport $[3^3, 1^5; 2^7]$ produces a unique tree with monodromy group $G \cong A_7 \wr \mathbb{Z}_2$.
\end{pro}
\begin{proof} 
This is a sporadic case that may be verified by hand.
\end{proof}

\section{Future Directions}

The reader will notice that there are several obvious pathways left open by this paper. At the very least, the authors would like to see Table \ref{bigtable} completed with the missing entry. That is, what is the monodromy group for the tree with $n$ edges and passport $[s^{r-1},t;r, 1^{(r-1)(s-1)+(t-1)}]$ where $r>1$ is even and $\gcd\left({\gcd(s,t)},{\frac{n}{\gcd(s,t)}}\right)\neq 1$? Incidentally, another unresolved curiousity in Table \ref{bigtable} is that since each entry refers to a tree with passport size one, for each entry there should exist a Shabat polynomial with rational coefficients. However, we were not able to accurately ``rationalize'' the Shabat polynomial given for the tree with passport $[r^2,1^{4r-3};3^{2r-1}]$. 

As for another direction of further inquiry, we note that the present paper focuses exclusively on (planar) trees with passports of size one. However, we know that there exists an exhaustive list of trees of passport size two, and perhaps there are other such lists for passports of even greater size \cite{shabat_plane_1994}. At the very least, it would be interesting to see the complete list of monodromy groups for trees with passport size two in comparison with the completion of Table \ref{bigtable}. Finally, it would also be of interest to see similar results for classes of dessins having at least one cycle and/or for dessins with genus greater than one.

\section{Acknowledgements}
We thank the Willamette University Mathematics Consortium REU for providing a beautiful working environment, as well as the generous support of NSF Grant \#1460982. We also wish to thank Edray Goins for his helpful input on the content of this paper. 
\nocite{*}
\bibliographystyle{plain}
\bibliography{main}

\end{document}